\documentclass[a4paper, 11pt]{amsart}   
\usepackage{mathptmx, amssymb,amscd,latexsym, eulervm,calligra,mathrsfs,amsmath,amsthm,amsfonts,amsrefs,tikz}
\usepackage[colorlinks=true,linkcolor=blue,urlcolor=blue,citecolor=blue]{hyperref}
\usepackage{xurl}
\hypersetup{breaklinks=true}

\usepackage{setspace}
\usepackage{tabularx}
\usepackage{comment}
\usepackage{amsfonts}
\usepackage{paralist}
\usepackage{aliascnt}
\usepackage{blkarray}
\usepackage{booktabs}
\usepackage{enumitem}
\usepackage{setspace}
\usepackage[inner=2.5cm,outer=2.5cm, bottom=3.2cm]{geometry}
\usepackage{hyperref}
\usepackage[capitalize,nameinlink]{cleveref}
\usepackage{tikz, tikz-cd}
\usepackage{calligra,mathrsfs}
\usetikzlibrary{matrix}
\usetikzlibrary{arrows,calc}
\usepackage{url}
\usepackage{amssymb,latexsym,amsmath,amsfonts,amsthm,amscd,graphicx,url,color,hyperref,stmaryrd}

\theoremstyle{plain}

\newtheorem{thm}{Theorem}[section]
\newtheorem{prop}[thm]{Proposition}
\newtheorem{cor}[thm]{Corollary}
\newtheorem{lemma}[thm]{Lemma}

\theoremstyle{definition}

\newtheorem{remark}[thm]{Remark}

\newtheorem{example}[thm]{Example}
\newtheorem{problem}[thm]{Problem}
\newtheorem{conjecture}[thm]{Conjecture}

\newcommand{\scR}{\mathscr{R}}
\newcommand{\scT}{\mathscr{T}}
\newcommand{\scC}{\mathscr{C}}
\newcommand{\scD}{\mathscr{D}}

\newcommand{\G}{\Gamma}
\newcommand{\edim}{\mathrm{edim}}
\newcommand{\mult}{\mathrm{mult}}
\newcommand{\type}{\mathrm{type}}

\newcommand{\mm}{\mathbf{m}}
\newcommand{\NN}{\mathbb{N}}
\newcommand{\ZZ}{\mathbb{Z}}
\newcommand{\gr}{\mathrm{gr}}
\newcommand{\Ap}{\mathrm{Ap}}

\newcommand{\kk}{\Bbbk}

\newcommand{\bb}{\mathbf{b}}

\newcommand{\Tor}{\mathrm{Tor}}
\newcommand{\HF}{\mathrm{HF}}

\newcommand{\mC}{\mathfrak{C}}

\newcommand{\hI}{\hat{I}}
\newcommand{\hL}{\hat{L}}
\newcommand{\hS}{\hat{S}}
\newcommand{\hC}{\hat{C}}

\makeatletter
\@namedef{subjclassname@2020}{
  \textup{2020} Mathematics Subject Classification}
\makeatother

\begin{document}

\author[A.\,Moscariello, A.\,Sammartano]{Alessio~Moscariello and Alessio~Sammartano}
\address{Alessio Moscariello: Dipartimento di Matematica e Informatica\\Università degli Studi di Catania\\Catania\\Italy}
\email{alessio.moscariello@unict.it}
\address{Alessio Sammartano: Dipartimento di Matematica \\ Politecnico di Milano \\ Milan \\ Italy}
\email{alessio.sammartano@polimi.it}

\title{On minimal presentations of numerical monoids}


\begin{abstract}
We consider the classical problem of determining the largest possible cardinality of a minimal presentation of a numerical monoid with given embedding dimension  and multiplicity.
Very few  values of this cardinality are known.
In addressing this problem, 
we apply tools  from Hilbert functions and free resolutions 
of artinian standard graded algebras. 
This approach allows us to solve the problem in many cases and, at the same time,
identify subtle difficulties 
in the remaining cases. 
As a by-product of our analysis, we deduce results  for the corresponding problem for the type of a numerical monoid.
\end{abstract}

\maketitle

\section{Introduction}

How many relations can a numerical monoid have?
Despite the  simplicity of the objects involved, this question is still not well understood.
This paper is devoted to finding sharp upper bounds for the number of minimal relations of a numerical monoid.
Our results will indicate a surprising complexity behind this apparently simple question.

In order to formulate precise questions and results, let us introduce some terminology.
A numerical monoid $\G\subseteq \NN$ is a cofinite additive submonoid of $\NN$.
As with many other objects in algebra, monoids  are classically studied by means of their generators and relations.
Every numerical monoid has a unique minimal set of  generators $g_0 < g_1 < \cdots < g_e \in \G$, 
and we write $\G= \langle g_0, g_1, \ldots, g_e\rangle$ to denote this fact.
The cardinality $e+1$ of the minimal set of generators is called the embedding dimension of $\G$, 
and it is denoted by $\edim(\G)$.
The monoid morphism  $\varphi: \NN^{e+1} \to \NN$ defined by $\varphi(a_0, a_1 , \ldots, a_e) = \sum_{i=0}^e a_i g_i$ yields an isomorphism  $\G \cong \NN^{e+1}/\ker(\varphi)$,
where $\ker(\varphi) = \{(\mathbf{a},\mathbf{b}) \in \NN^{e+1} \times \NN^{e+1} \, \mid \, \varphi(\mathbf{a}) = \varphi(\mathbf{b})\}$ is the kernel congruence of $\varphi$.
A minimal presentation of $\G$ is a minimal set of generators of the congruence $\ker(\varphi)$.
Minimal presentations are not unique, 
but their cardinality depends only on $\G$.
This cardinality is denoted by $\rho(\G)$, and it is referred to as the number of minimal relations of $\G$.
For example, 
the monoid $\G = \langle 10, 13, 15\rangle$ has a minimal presentation $\big\{\big((3,0,0),(0,0,2)\big), \big((0,5,0),(2,0,3)\big)\big\}$, 
and thus $\rho(\G) = 2$.
We refer to \cite{RosalesGarciaSanchezBook} for  details on numerical monoids and their minimal presentations.

There is a well-understood sharp lower bound for the number of minimal relations:
we always have $\rho(\G) \geq \edim(\G)-1$,
and equality is attained precisely by the complete intersection numerical monoids.
On the other hand, the picture for the upper bound is more interesting. 
It is known that $\rho(\G) \leq 3$ if $\edim(\G) \leq 3$
\cite{Herzog},
but if $\edim(\G) \geq 4$ then $\rho(\G)$ cannot be bounded by any function of the embedding dimension
\cite{Bresinsky}.
For this reason, one considers another invariant of $\G$:
the multiplicity of $\G$ is the smallest generator $g_0$, and it is denoted by $\mult(\G)$.
We always have $\edim(\G) \leq \mult(\G)$.

In this paper, we investigate the following problem.

\begin{problem}\label{ProblemNumberRelations}
Determine 
$
\mathscr{R}(e,m) := \max \big \{ \rho(\G) \, \mid \,  \edim(\G) = e+1,\, \mult(\G) = m\big\}.
$
\end{problem}

The problem also appears in \cite[Question 7.2]{EHOOPK}.
An   upper bound for  $\mathscr{R}(e,m)$ was proved in
\cite[Theorem 2.3]{Rosales96}, showing, in particular, that $\mathscr{R}(e,m)$ is  well-defined.
However, this bound is  not sharp. 
Only a few values of $\mathscr{R}(e,m)$  are known, namely,  when the  difference $\mult(\G)-\edim(\G)$ is small.
The case   $\mult(\G)=\edim(\G)$  is well-known, 
and we have $\scR(e,e+1) = {e+1 \choose 2}$.
Rosales and García-Sánchez proved in \cite{RosalesGarciaSanchez98} that $\scR(e,e+2) =  \scR(e,e+3) ={e+1 \choose 2}$.
However,
they also observed that the problem becomes more complicated as the gap $\mult(\G)-\edim(\G)$ widens, and suggested that alternative methods are needed to tackle the general case, cf. \cite[Section 1.3]{RosalesGarciaSanchez98}. 

Our goal is to determine the value of $\scR(e,m)$ in many more cases, and, in general, to understand the nature of the difficulties of Problem \ref{ProblemNumberRelations}.
To state our results, 
we introduce some arithmetic definitions related to   the Macaulay expansion of an integer.
Given positive integers $n$ and $d$,
there are uniquely  determined integers  $n_d > n_{d-1} > \cdots > n_j \geq j \geq 1$
such that  $n = {n_d \choose d} + {n_{d-1}\choose d-1} + \cdots + {n_j \choose j}$. 
Define 
\begin{equation*}\label{EqDefUpperMacaulay}
n^{\langle d\rangle}= \binom{n_d+1}{d+1}+\binom{n_{d-1}+1}{d}+\cdots+\binom{n_j+1}{j+1}.
\end{equation*}
Let $e,m \in \NN$ be such that $1 < e < m$.
Let $r$ be the unique integer such that ${e+r-1 \choose r-1} \leq m < {e+r \choose r}$,
and let $s = m-{e+r-1 \choose r-1}$.
 Define
\begin{equation}\label{EqDefCem}
\scC(e,m) := {e+r-1 \choose r} +s^{\langle r\rangle}-s.
\end{equation}
It follows from  \cite[Theorem 4.7]{ERV}  that $\scR(e,m) \leq \scC(e,m)$ for all $e,m$.

The main result of this paper is the following: 

\begin{thm}\label{MainTheoremNumberRelations}
Let $e,m \in \NN$ be such that $3 \leq e < m$.
\begin{enumerate}
\item 
For each fixed value of $m-e$, we have 
  $\scR(e,m) = \scC(e,m)$ for all $m \gg 0$.
\item 
For each fixed value of $e$, we have $\scR(e,m) < \scC(e,m)$ for $m \gg 0$.
\end{enumerate}
\end{thm}

We also show that $\scR(e,m) = \scC(e,m)$ whenever $m-e\leq 6$,  improving thus the aforementioned results of \cite{RosalesGarciaSanchez98}.
On the other hand, if $\mult(\G) -\edim(\G) = 7$, 
we show that there exist values of $e,m$ for which $\scR(e,m) < \scC(e,m)$:
for example, we have 
$\scR(3,10) < \scC(3,10)$.

An additional benefit of our analysis is that, 
while the main focus is on minimal presentations, 
it naturally yields results for the type of numerical monoids as well.
An integer $p \in \NN\setminus \G$ is called a pseudo-Frobenius number of $\G$ if $p+g \in \G$ for all $g\in \G\setminus\{0\}$.
The cardinality of the set of pseudo-Frobenius numbers is called the type of $\G$, and  is denoted by $\type(\G)$.
As before, it is known that $\type(\G) \leq 2$ if $\edim(\G) \leq 3$, but otherwise $\type(\G)$ cannot be bounded by a function of $\edim(\G)$
\cite{FGH}.
Then, one can consider the version of Problem \ref{ProblemNumberRelations} for the type.

\begin{problem}\label{ProblemType}
Determine 
$
\mathscr{T}(e,m) := \max \big \{ \type(\G) \, \mid \,  \edim(\G) = e+1,\, \mult(\G) = m\big\}.
$
\end{problem}

\noindent

It turns out that the function $\mathscr{T}(e,m)$ exhibits a similar behavior to $\scR(e,m)$,
though the problem becomes somewhat more complicated.
With the same notation introduced above, 
define

\begin{equation*}\label{EqDefLowerMacaulay}
n_{\langle d\rangle}= \binom{n_d-1}{d}+\binom{n_{d-1}-1}{d-1}+\cdots+\binom{n_j-1}{j}
\end{equation*}
and
\begin{equation}\label{EqDefDem}
\scD(e,m) := {e+r-2 \choose r-1} +s_{\langle r\rangle}.
\end{equation}
Then, we have $\scT(e,m) \leq \scD(e,m)$ by \cite[Theorems 4 and 6]{EGV}, and we prove the following:

\begin{thm}\label{MainTheoremType}
Let $e,m \in \NN$ with $3 \leq e < m$.
For each fixed value of $m-e$, we have 
  $\scT(e,m) = \scD(e,m)$ for all $m \gg 0$.
\end{thm}

As before,
we also obtain $\scT(e,m) = \scD(e,m)$ for $m-e\leq 6$.
However, proving
 the  strict inequality $\scT(e,m) < \scD(e,m)$ is harder than $\scR(e,m) < \scC(e,m)$.
Our current methods do not yield  the analogue of the second part of Theorem \ref{MainTheoremNumberRelations} for the type,
though they strongly suggest it holds.
Just to illustrate the complexity, we prove the strict inequality in one case,
showing that $\scT(4,11) = \scD(4,11)-1$.

Let us now discuss the ideas and methods in this work. 
The main techniques are drawn from commutative algebra, together with some ingredients from additive combinatorics and factorization theory.
While it is well-known that numerical monoids can be studied in the framework of commutative algebra, 
through their one-dimensional local monoid rings,
our  approach is to shift the focus to zero-dimensional standard graded  algebras and special monomial ideals.
There is a rich theory of Hilbert functions and free resolutions of such algebras and ideals, 
and we take full advantage of these tools.
This point of view offers a more transparent explanation for  the nature of the  results and of the convoluted arithmetic formulas above.
For example, for each pair $(e,m)$ there is a distinguished monomial ideal  in $e$  variables and with colength $m$ that attains the largest total Betti numbers among all perfect ideals with the same multiplicity and codimension \cite{Valla}.
The formulas \eqref{EqDefCem} and \eqref{EqDefDem} are respectively equal to the first and last Betti number of this ideal. 

Part (1) of Theorem \ref{MainTheoremNumberRelations} shows that, 
if the generators of $\G$ are given enough room, 
toric ideals of numerical monoids behave like arbitrary  ideals.
To prove this result, we employ extremal sumsets to construct monoids with Apéry sets with unique factorization, 
and use this feature to produce the desired invariants.
This suffices for solving Problem \ref{ProblemNumberRelations} in many cases. 
In the remaining cases, the combinatorics of numerical monoids gives rise to  new constraints, 
which lead to tighter bounds for the number of relations.
Algebraically, these constraints may arise for a variety of different reasons, 
such as the presence of polynomial relations of low degree, sub-extremal Hilbert functions, or regular sequences of low degrees.
We employ some finer tools that take  this additional data into account,
together with further combinatorial results on sumsets,
to obtain part (2) of Theorem \ref{MainTheoremNumberRelations}.
In some cases,
a famous  problem comes into play:
the Eisenbud-Green-Harris Conjecture.
We use partial results on this conjecture  to show that sharper bounds than $\scC(e,m)$ do  sometimes exist. 
The Eisenbud-Green-Harris Conjecture is widely open, and it is reasonable to expect that its validity may imply sharper bounds in several other cases.
In our opinion,  this suggests that a complete solution to Problem \ref{ProblemNumberRelations} might be out of reach with the current knowledge.

Theorem \ref{MainTheoremType} is deduced from  the analogous statement of Theorem \ref{MainTheoremNumberRelations}  applying certain rigidity properties of free resolutions with extremal Betti numbers. 
Some new complications arise when considering the failure of  the bound $\scT(e,m) \leq \scD(e,m)$  to be sharp.
We only prove non-sharpness in one case, using   results on the Lex-Plus-Powers Conjecture,
which is a refinement of the Eisenbud-Green-Harris Conjecture.
Nevertheless, we conjecture that the bound fails to be sharp in many other cases. 

We remark here that similar techniques have proved to be useful also to study 
of another conjecture about the width of a numerical semigroup.
We refer the interested reader to the paper  \cite{CMS} by Caviglia and the authors.

\section{Algebraic toolbox}\label{SectionToolbox}

In this section, 
we collect the tools from commutative algebra that will be applied in the next sections to the study of minimal presentations and type of numerical monoids. 
We will review the  relevant literature and prove some new technical results along the way.
At the end of this section, we give a brief overview of some rings and ideals associated to a numerical monoid.

\subsection{Betti numbers}
We begin by reviewing some basic notions of commutative algebra.
A good reference for these topics is \cite{HerzogHibi}.

Let $\kk$ denote an arbitrary field. 
If $V$ is a $\ZZ$-graded $\kk$-vector space, we denote by $[V]_d$ its graded component of degree $d$.
Assuming $\dim_\kk[V]_d < \infty $ for all $d$,
the Hilbert function of $V$ is the function $\HF(V):\mathbb{Z}\to \mathbb{N}$ defined by $\HF(V,d)=\dim_\kk[V]_d$.

For the rest of this section, let $S = \kk[x_1, \ldots, x_e]$ denote the polynomial ring.
If $M$ is a finitely generated graded $S$-module, 
its graded Betti numbers  are 
$\beta^S_{i,j}(M) = \dim_\kk[\Tor^S_i(M,\kk)]_j$,
where $i \in \NN, j \in \ZZ$.
The integer $\beta_{i,j}^S(M)$ is the number of generators of degree $j$ in the $i$-th free module in a minimal free resolution of $M$ over $S$.
In particular, $\beta^S_{0,j}(M)$ is the number of minimal generators of $M$ of degree $j$.
The total Betti numbers of $M$ are $\beta^S_i(M) = \sum_{j \in \ZZ}\beta^S_{i,j}(M)$,
so $\beta_0^S(M)$ is the number of minimal generators of $M$.
By Hilbert's Syzygy Theorem, we have $\beta_i^S(M)=0$ for all $i > e$.
We denote the length of $M$ as an $S$-module by $\ell(M)$;
 observe that it is the same as the vector space dimension $\dim_\kk M$.
If $I \subseteq S$ is an ideal, the length $\ell(S/I)$ is also referred to as the colength of $I$.

A lexsegment ideal of $S$ is a monomial ideal  $L\subseteq S$ such that, for every $d$, the vector space $[L]_d$ is spanned by an initial segment of monomials of $S$, where monomials are ordered lexicographically.
By Macaulay's theorem \cite[Theorem 6.3.1]{HerzogHibi}, there is a bijection  between Hilbert functions of graded quotients of $S$ and lexsegment ideals of $S$.
The following theorem of Bigatti, Hulett and Pardue  asserts that lexsegment ideals achieve the largest graded Betti numbers among all ideals with the same Hilbert function,
see \cite[Theorem 7.3.1]{HerzogHibi} or \cite[Theorem 31]{Pardue}.

\begin{thm}\label{TheoremBHP}
Let $I\subseteq S$ be a homogeneous ideal, and $L\subseteq S$ be the unique lexsegment ideal such that $\HF(I) = \HF(L)$.
Then, we have
$\beta^S_{i,j}(I) \leq \beta^S_{i,j}(L)$ for all $i,j$.
\end{thm}

Let $\mm= (x_1, \ldots, x_e)$ denote the homogeneous maximal ideal of $S$.
An ideal $C \subseteq S $ is very compressed ideal if $\mm^{r+1} \subsetneq C \subseteq \mm^r$ for some $r \in \NN$.
For each $m \in \NN$, there exists a unique very compressed lexsegment ideal $C \subseteq S$ such that $\ell(S/C) =m$.
We denote this ideal by $\mC(m,S)$.
The following theorem of Valla \cite[Theorem 4]{Valla}
asserts that  very compressed lexsegment ideals achieve the 
largest total Betti numbers among all ideals with the same colength.
See also \cite{ERV,EGV,CavigliaSammartano}.

\begin{thm}\label{TheoremValla}
Let $I\subseteq S$ be a homogeneous ideal such that $\ell(S/I) =m < \infty$,
and let $C = \mC(m,S)$.
Then, we have
$\beta^S_{i}(I) \leq \beta^S_{i}(C)$ for all $i$.
\end{thm}

The total Betti numbers of $\mC(m,S)$ can be computed  in terms of $e,m,i$
\cite[Proposition 5]{Valla}.
The  formulas are explicit but rather convoluted;
however,
 the ones that are most relevant for our purposes are the first and last non-zero Betti numbers, which admit simpler formulas, already presented in the Introduction.

\begin{prop}\label{PropFirstLastBettiValla}
We have $\beta_0^S(\mC(m,S)) = \scC(e,m)$ and $\beta_{e-1}^S(\mC(m,S)) = \scD(e,m). $
\end{prop}
\begin{proof}
See \cite[p. 317]{Valla} or \cite[Proposition 4.2]{ERV} for $\beta_0$, 
\cite[p. 316]{Valla} or \cite[Theorem 6]{EGV} for $\beta_{e-1}$.
\end{proof}

It is natural to ask  to what extent  $\mC(m,S)$ is the unique ideal attaining the largest Betti numbers.
More precisely: if $I$ differs from $\mC(m,S)$ in  some  numerical invariant, such as the Hilbert function or the initial degree, are the inequalities in Theorem \ref{TheoremValla} strict?
This appears to be a fairly complicated problem in general, 
and nothing is known about it.
We are going to give a partial affirmative answer for the first two Betti numbers.

Denote $\hS=S/(x_e) \cong\kk[x_1, \ldots, x_{e-1}]$.
If $I \subseteq S$ is a homogeneous ideal, denote $\hI = \frac{I+(x_e)}{(x_e)}\subseteq \hS$.
It is easy to see that if $I$ is lexsegment, resp. very compressed, then so is $\hI$.
The total Betti numbers of a lexsegment ideal can be determined through the following recursive formula.

\begin{prop}\label{PropRecursiveFormula}
Let $L\subseteq S$ be a lexsegment ideal such that $\ell(S/L)  < \infty$.
Then, for all $i$ we have $\beta_i^S(L) = \beta^{\hS}_i(\hL) + \ell(\hS/\hL) {e-1 \choose i}$.
\end{prop}
\begin{proof}
See, for example, \cite[Proof of Lemma 4.2]{CMS} or \cite[Proof of Theorem 3.7]{CavigliaSammartano}.
\end{proof}

We will need some further properties of the ideals $\mC(m,S)$.

\begin{prop}\label{PropInclusionHyperplane}
The following properties hold:
\begin{enumerate}
\item if $m_1 < m_2$, then $\mC(m_2,S)\subseteq \mC(m_1,S)$;
\item if $I\subseteq S$ is a lexsegment ideal with $\ell(S/I) =m$ and  $C = \mC(m,S)$,
then $\ell(\hS/\hI) \leq \ell(\hS/\hC)$.
\end{enumerate}
\end{prop}
\begin{proof}
Property (1) follows immediately from the definition of $\mC(m,S)$,
while property (2) is a special case of \cite[Corollary 3.4]{CavigliaSammartano}.
\end{proof}

\begin{lemma}\label{LemmaBettiInequalityInclusion}
If $m_1 \leq m_2$, then $\beta_i^S(\mC(m_1,S)) \leq \beta_i^S(\mC(m_2,S))$ for all $i$.
\end{lemma}
\begin{proof}
We prove the statement by induction on $e$.
If $e=1$, then $S=\kk[x_1]$, and the claim is obvious, so assume $e > 1$.
Let $C_1=\mC(m_1,S)$ and $C_2=\mC(m_2,S)$.
By Proposition \ref{PropInclusionHyperplane} (1), we  have $C_2 \subseteq C_1$, and therefore $\widehat{C_2}\subseteq \widehat{C_1}$.
On one hand, this immediately implies $\ell(\hS/\widehat{C_1}) \leq \ell(\hS/\widehat{C_2})$.
On the other hand, since both ideals are very compressed and lexsegment,  by induction this implies
 $\beta_i^{\hS}(\widehat{C_1}) \leq \beta_i^{\hS}(\widehat{C_2})$ for all $i$.
The conclusion follows from the formula of Proposition \ref{PropRecursiveFormula}.
\end{proof}

\begin{prop}\label{PropLowerDegreeStrictInequality}
Let $I\subseteq S$ be a homogeneous ideal with $\ell(S/I) =m < \infty$.
Let $C = \mC(m,S)$ and let $r$ be the unique integer such that $\mm^{r+1} \subsetneq C \subseteq \mm^r$.
Assume  $[I]_{r-1} \ne 0$. 
Then, we have $\beta_i^S(I) < \beta_i^S(C)$ if $i=0$ or if $i=1$ and $e\geq 2$. 
\end{prop}
\begin{proof}
We prove the statement by induction on $e$.
If $e=1$ the statement is vacuous, since $C\subseteq S=\kk[x_1]$ is the unique homogeneous ideal of colength $m$.
Assume $e>1$. 
By Theorem \ref{TheoremBHP}, we may assume that $I$ is a lexsegment ideal.
By Proposition \ref{PropInclusionHyperplane} (2), we have $\ell(\hS/\hI)\leq \ell(\hS/\hC)$.
We distinguish two cases.

Suppose
$\ell(\hS/\hI)< \ell(\hS/\hC)$.
Let $C' = \mC(\ell(\hS/\hI),\hS)$, then $\beta^{\hS}_i(\hI)\leq \beta^{\hS}_i(C')$ by Theorem \ref{TheoremValla}.
We also have
 $\beta^{\hS}_i(C') \leq \beta^{\hS}_i(\hC)$  by Lemma \ref{LemmaBettiInequalityInclusion}.
By Proposition \ref{PropRecursiveFormula}, we obtain the desired conclusion.

Suppose 
$\ell(\hS/\hI)= \ell(\hS/\hC)$.
Observe that $[\hI]_{r-1}\ne 0$, since
 $x_1^{r-1}\in I$.
 If $e \geq 3$, 
since $\hC$ is very compressed and lexsegment,
we have $\beta^{\hS}_i(\hI) <  \beta^{\hS}_i(\hC)$ for $i=0,1$ by induction,
and by Proposition \ref{PropRecursiveFormula} we obtain the desired conclusion.
If $e=2$, 
so that $\hS=\kk[x_1]$,
the case $\ell(\hS/\hI)= \ell(\hS/\hC)$ cannot occur, since $[\hI]_{r-1}\ne 0$ forces $\hC \subsetneq \hI$ and thus $\ell(\hS/\hI)< \ell(\hS/\hC)$.
 \end{proof}
 
\begin{remark}\label{RemarkNonStrictHigherBetti}
We note that the statement of  Proposition \ref{PropLowerDegreeStrictInequality} is not true for all total Betti numbers. 
Inspecting the proof, the issue occurs when $\ell(\hS/\hI)= \ell(\hS/\hC)$ and $i=e-1$:
then $\beta^{\hS}_{e-1}(\hI)=\beta^{\hS}_{e-1}(\hC)=0$,
so $\beta^{S}_{e-1}(I)=\beta^{S}_{e-1}(C)$ by Proposition \ref{PropRecursiveFormula}.
For an explicit example, let $e=3$, $m=16$, $I= (x_1^2) + \mm^4$, and $C = \mC(16,S) = (x_1^3,x_1^2x_2,x_1^2x_3,x_1x_2^2)+\mm^4$.
Then, we have $\beta_2^S(I) = \beta_2^S(C) = 7$.
\end{remark}

We  now establish a rigidity property of the resolution of $\mC(m,S)$.
This  should be regarded as the ``colength counterpart'' of 
the analogous result for lexsegment ideals proved in \cite[Corollary 1.4]{HerzogHibiComponentwise}.

\begin{prop}\label{PropRigidity}
Let $I\subseteq S$ be a homogeneous ideal with $\ell(S/I) =m < \infty$, and let
 $C = \mC(m,S)$.
If $\beta_0^S(I) = \beta_0^S(C)$,
then $\beta_i^S(I) = \beta_i^S(C)$ for all $i$.
\end{prop}
\begin{proof}
Let $L \subseteq S $ denote the unique lexsegment ideal with $\HF(L)=\HF(I)$.
Combining the hypothesis with Theorems \ref{TheoremBHP} and \ref{TheoremValla}, 
we obtain  $\beta_0^S(I) = \beta_0^S(L)=  \beta_0^S(C)$.
Homogeneous ideals $I$ for which the first equality holds are known as \emph{Gotzmann ideals}
\cite{HerzogHibiComponentwise,MuraiHibi},
and they satisfy $\beta_i^S(I) = \beta_i^S(L)$ for all $i$,
cf.
 \cite[Corollary 1.4]{HerzogHibiComponentwise}.
It remains to show that $\beta_i^S(L) = \beta_i^S(C)$ as well.
As before, we proceed by induction on $e$, and the case $e=1$ is trivial, so assume $e>1$.
By Proposition \ref{PropInclusionHyperplane} (2), we have $\ell(\hS/\hL)\leq \ell(\hS/\hC)$.
Letting $C' = \mC(\ell(\hS/\hL),\hS)$, we have 
$\beta^{\hS}_i(\hL)\leq \beta_i^{\hS}(C')$ by Theorem \ref{TheoremValla}, 
and
$\beta^{\hS}_i(C') \leq \beta^{\hS}_i(\hC)$ by Lemma \ref{LemmaBettiInequalityInclusion}.
Suppose that $\ell(\hS/\hL)< \ell(\hS/\hC)$, then Proposition \ref{PropRecursiveFormula}
yields $\beta_0^S(L) = \beta_0^{\hS}(\hL) + \ell(\hS/\hL) < \beta^{\hS}_0(\hC) + \ell(\hS/\hC) = \beta_0^S(C)$,
contradiction.
Thus, we must have  $\ell(\hS/\hL)= \ell(\hS/\hC)$ and,
for the same reason, $\beta_0^{\hS}(\hL)=\beta^{\hS}_0(\hC)$.
By induction, we deduce that $\beta_i^{\hS}(\hL)=\beta^{\hS}_i(\hC)$,
and another application of  Proposition \ref{PropRecursiveFormula} concludes the proof.
 \end{proof}
 
Finally, 
we discuss two long-standing conjectures 
that attempt to  refine the theorems of Macaulay and of  Bigatti, Hulett, and Pardue,  by taking into account subtler data than just the Hilbert function. 
If $ f_1, \ldots, f_c \in S$ is a regular sequence of homogeneous polynomials,
its degree sequence is the vector $\mathbf{d}= (\deg(f_1), \ldots, \deg(f_c))$.
We always require that  $\deg(f_1) \leq \cdots \leq \deg(f_c)$.

\begin{conjecture}[Eisenbud-Green-Harris Conjecture]\label{EGH}
Let $I\subseteq S$ be a homogeneous ideal containing a regular sequence of degree sequence $(d_1, \ldots, d_c)$.
If there exists a lexsegment ideal $L\subseteq S$ such that $\HF(I) = \HF(L + (x_1^{d_1}, \ldots, x_c^{d_c}))$,
then $\beta_{0,j}^S(I)\leq \beta_{0,j}^S(L + (x_1^{d_1}, \ldots, x_c^{d_c}))$ for all $j$.
\end{conjecture}

\begin{conjecture}[Lex-Plus-Powers Conjecture]
Let $I\subseteq S$ be a homogeneous ideal containing a regular sequence of degree sequence $(d_1, \ldots, d_c)$.
If there exists a lexsegment ideal $L\subseteq S$ such that $\HF(I) = \HF(L + (x_1^{d_1}, \ldots, x_c^{d_c}))$,
then $\beta_{i,j}^S(I)\leq \beta_{i,j}^S(L + (x_1^{d_1}, \ldots, x_c^{d_c}))$ for all $i,j$.
\end{conjecture}

We refer to them as the EGH and LPP Conjectures. 
The EGH Conjecture is usually stated in a different way, which closely follows the statement of Macaulay's theorem:
if $I\subseteq S$ is a homogeneous ideal containing a regular sequence of degree sequence $(d_1, \ldots, d_c)$,
there exists a lexsegment ideal $L\subseteq S$ such that $\HF(I) = \HF(L + (x_1^{d_1}, \ldots, x_c^{d_c}))$.
It is not hard to see that the two formulations are equivalent, for each fixed degree sequence.
The formulation given in Conjecture \ref{EGH} will be more convenient for our purposes.
We also remark that, while the two conjectures above focus on Hilbert functions and graded Betti numbers,
a  version for colength and total Betti numbers can also be formulated: 
in fact,  if the LPP Conjecture holds, then it implies a corresponding refinement of Valla's Theorem \ref{TheoremValla}, see 
\cite[Proposition 5.4]{CavigliaSammartano}.
We refer to  \cite{CDSS21,FrRi07,Gu21} for more details on these two problems.

We record here the following computation, needed later for applications related to the two conjectures above, see
Propositions \ref{Prop411and512} and \ref{PropType411}.

\begin{lemma}\label{Lemma1e6}
Let $e \in \{4,5\}$.
Let $I \subseteq S $ be a homogeneous ideal with $\ell(S/I) = e+7$ and $C=\mC(e+7,S)$.
If $\HF(I) \ne \HF(C)$, then $\beta^S_i(I)< \beta^S_i(C)$ for all $i = 0, \ldots, e-1$.
\end{lemma}

\begin{proof}
By Theorem \ref{TheoremBHP}, we may assume that $I$ is a lexsegment ideal.
There are finitely many lexsegment ideals of a given colength, 
and the statement can be verified directly using a computer algebra system such as Macaulay2 \cite{M2}.
\end{proof}

\subsection{Monoid algebras}\label{SubsectionMonoidAlgebras}
We briefly recall here some well-known facts about numerical monoids and their monoid algebras.

The Apéry set of a numerical monoid $\G$ is the set $\Ap(\G) = \{ g \in \G \, \mid \, g-m \notin \G\}$.
We have $| \Ap(\G) | = m$ and $0, g_1, \ldots, g_e \in \Ap(\G)$. 

Let $\G = \langle g_0, \ldots, g_e \rangle$ be a numerical monoid and let $z$ be an indeterminate.
The monoid algebra of $\G$ is $R_\G = \Bbbk\llbracket z^{g_0}, \ldots, z^{g_e} \rrbracket\subseteq
\Bbbk\llbracket z\rrbracket$.
We have a presentation
$R_\G = P / {I}_\G$,
where $P = \Bbbk \llbracket x_0,  \ldots, x_e \rrbracket$ is the power series ring and ${I}_\G$ is the toric ideal of $\G$,
generated by all the binomials $x_0^{a_0}\cdots x_e^{a_e}-x_0^{b_0}\cdots x_e^{b_e}$ such that 
$\sum_{i=0}^e a_i g_i = \sum_{i=0}^e b_i g_i$.
We have 
$\rho(\G) = \beta_0^P({I}_\G)$
and
$\type(\G) = \beta_{e-1}^P({I}_\G).$
The ring $R_\G $ is  a one-dimensional local domain of multiplicity  $m=g_0$,
and the quotient $\overline{R}_\G = R_\G/(z^m)$ is an artinian local ring of length $m$.
We have $\overline{R}_\G = \overline{P}/\overline{I}_\G$,
where
$\overline{P}=P/(x_0) \cong \Bbbk \llbracket x_1, \ldots, x_e \rrbracket$
and
 $\overline{I}_\G = \frac{I_\G+(x_0)}{(x_0)}$. 

We consider the associated graded ring $\gr(\overline{R}_\G) = \bigoplus_{r \geq 0} \frac{\mathfrak{n}^{r}}{\mathfrak{n}^{r+1}}$, where $\mathfrak{n}\subseteq \overline{R}_\G$ denotes the unique maximal ideal.
Let $S = \Bbbk[ x_1, \ldots, x_e]$,
then we have a presentation 
$\gr(\overline{R}_\G) = S/\overline{I}_\G^{\ *}$, 
where $\overline{I}_\G^{\ *}\subseteq S$ is the ideal of initial forms of $\overline{I}_\G$.
The latter is a homogeneous ideal, generated by binomials and monomials
\begin{align}\label{EqIdealIstar}
\begin{split}
\overline{I}_\G^{\ *} = & \left(
x_1^{a_1}\cdots x_e^{a_e}-x_1^{b_1}\cdots x_e^{b_e}
\, \mid \, \sum_{i=1}^e a_i g_i = \sum_{i=1}^e b_i g_i  \in \Ap(\G)
\,\text{ and }\,
\sum_{i=1}^e a_i = \sum_{i=1}^e b_i
\right)\\
& +
\left(
x_1^{a_1}\cdots x_e^{a_e}
\, \mid \, 
\sum_{i=1}^e a_i g_i = \sum_{i=1}^e b_i g_i \in \Ap(\G), \sum_{i=1}^e a_i < \sum_{i=1}^e b_i
\text{ for some } b_1, \ldots, b_e \in \NN
\right)\\
&+
\left(
x_1^{a_1}\cdots x_e^{a_e}
\, \mid \, 
\sum_{i=1}^e a_i g_i \notin \Ap(\G)
\right).
\end{split}
\end{align}
For every $i \in \NN$, we have
\begin{equation}\label{InequalityBettis}
\beta^P_i(I_\G) = \beta^{\overline{P}}_i(\overline{I}_\G)\leq \beta^S_i(\overline{I}_\G^{\ *}).
\end{equation}

\section{Regions where the bound is sharp}\label{SectionRelationsSharp}

This section is devoted to proving the first part of Theorem \ref{MainTheoremNumberRelations}. 

Let  $\NN^e$ be the free commutative monoid with  canonical basis  $\bb_1, \ldots, \bb_e$.
A factorization of an element $g \in \Ap(\G)$ is a vector $c_1 \bb_1+ \cdots + c_e \bb_e \in \NN^e$ such that
$g = c_1 g_1+ \cdots + c_e g_e$.

We will use the following result of Rosales to compute the number of  relations.
We consider the   component-wise partial order in $\NN^e$.

\begin{prop}[\protect{\cite[Theorem 1]{Rosales99}}]\label{RelationsRosales}
Suppose that every element of $\Ap(\G)$ has a unique factorization.
Then, $\rho(\G)$ is equal to the number of minimal elements in the subset
$$
\mathcal{A}(\G) := 
\big\{
( a_1,\ldots, a_e) \in \NN^{e} \ | \  a_1g_1+\cdots+ a_eg_e \not \in \Ap(\G)
\big\} \subseteq \NN^{e}.
$$
\end{prop}

Using Proposition \ref{RelationsRosales} and Theorem \ref{TheoremValla}, we can prove the following result.

\begin{thm}\label{TheoremSharpLargeM}
Let $e, m \in \NN $ with $3 \leq e < m$.
For each fixed value of $m-e$, we have 
  $\scR(e,m) = \scC(e,m)$ for all $m \gg 0$.
\end{thm}

\begin{proof}
Let $\delta  = m -e$.
Since $\binom{t}{2}=\binom{t-1}{2}+t-1$, 
there exist unique  $t,q\in \NN$ such that $\delta - 1 = \binom{t}{2}-q$ and  $0 \le q \le t-2$.
The goal is to show that, for every $m \geq 2^t$, 
there exists a numerical monoid $\G_{m}$ such that $\mult(\G_m) = m, \edim(\G_m) = e+1 = m - \delta +1$,
and $\rho(\G_m) = \scC(e,m)$.
Since $\scR(e,m) \leq \scC(e,m)$ by \eqref{InequalityBettis}, Theorem \ref{TheoremValla}, and Proposition \ref{PropFirstLastBettiValla}, this implies the statement of the theorem.

We are going to show  that the numerical monoids
$$
\G_{m}=\big\{0\big\} \cup \big\{m+2^i-1 \,\mid\, i=0,\ldots,t-1\big\} \cup \big\{g \,\mid\, g \ge m+2^{t-1}+2^{t-q-1}-1\big\}\subseteq \NN
$$
satisfy the desired conditions.
Clearly, we have $\mult(\G_m) = m$.

Recall that $g_i$ denotes the $i$-th minimal generator of $\G_m$,
so $g_i = m+2^i-1$ if $i\leq t-1$.
We claim that the non-generator positive elements of $\Ap(\G_m)$ are exactly the elements
 $g_i + g_j$ with $1\leq i\leq j \leq t-2$ or $j = t-1, 1\leq i \leq t-q-1$.
Let $w \in \Ap(\G_m)$ and suppose that $w$ is positive and it is not a minimal generator of $\G_m$.
Then, 
\begin{equation}\label{EqInequalitiesPowers2}
w \leq \max \Ap(\G_m) = m + \max(\NN \setminus \G_m) = 2m + 2^{t-1}+2^{t-q-1}-2 \leq 2m + 2^t-2 < 3m,
\end{equation}
hence,
we must have $w = g_i + g_j$ for some $1 \leq i \leq j \leq e$.
Since $g_i > m$, the first inequality of \eqref{EqInequalitiesPowers2} yields 
$$
g_j \leq g_i+g_j - m = w -m \leq m+ 2^{t-1}+2^{t-q-1}-2,
$$
thus, we necessarily have $g_j = m+2^j-1$ and, likewise,
$g_i = m + 2^i-1$.
Using \eqref{EqInequalitiesPowers2} again, we have 
$$
g_i + g_j = 2m +2^i + 2^j -2  \leq 2m + 2^{t-1}+2^{t-q-1}-2 
\Rightarrow
2^i + 2^j   \leq 2^{t-1}+2^{t-q-1}
$$
thus, either $i, j \leq t-2$ or $j = t-1, i \leq t-q-1$.
Conversely,
if either $i, j \leq t-2$ or $j = t-1, i \leq t-q-1,$ then 
$
g_i+g_j-m = m + 2^i+2^j-2 < m+2^{t-1}+2^{t-q-1}-1\notin \Gamma,$
hence, $g_i + g_j \in \Ap(\G_m)$.
This concludes the proof of the claim.
Moreover, every  $g_i + g_j \in \Ap(\G_m) $ has a unique factorization: 
if $g_i + g_j =g_{i'}+g_{j'}$, 
then $2^i+2^j=2^{i'}+2^{j'}$, and, since the base 2 representation is unique, we must have $(i,j)=(i',j')$.
We conclude that
$$
\mult(\G_m) -
\edim(\G_m) =   {t-1 \choose 2} + (t-q-1 ) = {t \choose 2} -q = \delta-1,
$$
that is, $\edim(\G_m) = m-\delta+1 = e+1$, as desired.

By the above description of $\Ap(\G_m)$, in the notation of Proposition \ref{RelationsRosales}, we have 
$$
\mathcal{A}(\G_m) = \NN^{e} \setminus \big(
\{\mathbf{0}, \bb_1, \ldots, \bb_e\} \cup \{ \bb_i + \bb_j \,\mid \, 1 \le i \leq j \leq  t-2 \text{ or } j=t-1, 1 \le i \le t-q-1\}
\big).
$$
It follows that set of  minimal elements of $\mathcal{A}(\G_m)$ is $\mathcal{A}_1 \cup \mathcal{A}_2 \cup \mathcal{A}_3 \cup \mathcal{A}_4$, where
\begin{align*}
&\mathcal{A}_1 = \{ \bb_i + \bb_j \,\mid \, 1 \le i \le j, t\leq j \leq e\},  
&\mathcal{A}_2 =   \{ \bb_i + \bb_j \,\mid \,  t-q\le i \le t-1, j = t -1\},\\
&\mathcal{A}_3 = \{ \bb_i + \bb_j + \bb_k \, \mid \, 1 \leq i \leq j \leq k \leq t-2\},  
&\mathcal{A}_4 =   \{\bb_i + \bb_j + \bb_k \, \mid \, 1 \leq i \leq j  \leq t-q-1, k = t-1\}.
\end{align*}
Using Proposition \ref{RelationsRosales}, we conclude that
\begin{align*}
\rho(\G_m) &= |\mathcal{A}_1|+|\mathcal{A}_2|+|\mathcal{A}_3|+|\mathcal{A}_4|
=\sum_{j=t}^e j + (t-1)-(t-q-1) +{t \choose 3}+{t-q \choose 2} 
\\
&= {e+1 \choose 2}-{t \choose 2} + q+
{t \choose 3}+{t-q \choose 2}.
\end{align*}
Now, 
we determine the number $\scC(e,m)$, cf. \eqref{EqDefCem}.
First, we claim that $r = 2$, that is, 
${e+1 \choose 1} \leq m < {e+2 \choose 2}$.
The first inequality is trivial. 
Assume by contradiction that $m \geq {e+2 \choose 2}$.
Then, $\delta = m-e \geq  {e+2 \choose 2} -e = \frac{e^2+e+2}{2} > e$.
However, 
since $\delta \leq {t \choose 2} +1 \leq 2^{t-1}$ and $e + \delta = m \geq 2^t$, 
we deduce $e\geq \delta$, 
contradiction.

Thus, under the notation of \eqref{EqDefCem}, $r = 2$ and 
 ${e+1 \choose 1} \leq m < {e+2 \choose 2}$. 
Then,  
$$
s = m - {e+1 \choose 1} = \delta -1 = 
{t \choose 2} - q =  {t-1 \choose 2}+  {t-1-q \choose 1},
$$
therefore, $s^{\langle 2 \rangle} = {t \choose 3}+  {t-q \choose 2}$ and 
$$
\scC(e,m) = {e+r-1 \choose r} + s^{\langle r \rangle} -s =  {e+1 \choose 2} + {t \choose 3}+  {t-q \choose 2} - {t \choose 2} + q,
$$
	concluding the proof.
\end{proof}

Refining the proof strategy used in Theorem  \ref{TheoremSharpLargeM},
we obtain a complete solution of Problem \ref{ProblemNumberRelations} when the difference  $m-e$ is small.

\begin{thm}\label{TheoremSharpEM6}
Let $e, m \in \NN $ with $3 \leq e < m$.
If $m - e \leq 6$,
then  $\scR(e,m) = \scC(e,m)$.
\end{thm}

\begin{proof}
As in the proof of Theorem \ref{TheoremSharpLargeM}, it suffices to exhibit, for each pair $(e,m)$, a monoid $\G$ with $\mult(\G) = m, \edim(\G) = e+1$, and $\rho(\G) = \scC(e,m)$.
Let $\delta = m-e$.
For each $\delta = 1, \ldots, 6,$ consider the following numerical monoids:
\begin{enumerate}
\item $\G_{e,e+1}=\langle e+1,\ldots,2e+1 \rangle$,
\item $\G_{e,e+2}=\langle e+2,e+3,e+5,\ldots,2e+3\rangle$,
\item $\G_{e,e+3}=\langle e+3,e+4,e+6,e+7,e+9,\ldots,2e+5\rangle$,
\item $\G_{e,e+4}=\langle e+4,e+5,e+7,2e+13,2e+15, \ldots, 3e+11 \rangle$,
\item $\G_{e,e+5}=\langle e+5,e+6,2e+8,3e+12,3e+19,\ldots,4e+15 \rangle$,
\item $\G_{e,e+6}=\langle e+6,e+7,e+9,e+10,e+15,e+16,\ldots,2e+11\rangle$,
\end{enumerate} 
where the dots indicate an interval of consecutive integers.
In each case,
we have $\mult(\G_{e,e+\delta}) = e+\delta$ and $ \edim(\G_{e,e+\delta}) = e+1$.
	
We show that every element  $ g \in \Ap(\G_{e,e+\delta})$  that is positive and is not a generator of $\Gamma$ has a unique factorization.
We determine the Apéry sets explicitly, highlighting such elements $ g$ and determining their factorization.
\begin{enumerate}
\item $\Ap(\G_{e,e+1})=\{0,e+2,\ldots,2e+1\}$.
\item $\Ap(\G_{e,e+2})=\{0,e+3,e+5,\ldots,2e+3,\underline{2e+6}\}$. Unique factorization:
\\
$2e+6 = 2(e+3)=2g_1$.
\item $\Ap(\G_{e,e+3})=\{0,e+4,e+6,e+7,e+9,\ldots,2e+5,\underline{2e+8},\underline{2e+11}\}$. Unique factorizations:
\\
$2e+8=2(e+4)=2g_1$, 
$2e+11=(e+4)+(e+7)=g_1+g_3$.
\item $\Ap(\G_{e,e+4})=\{0,e+5,e+7,\underline{2e+10},\underline{2e+12},2e+13,\underline{2e+14}, 2e+15,\ldots,3e+11\}$.
\\
Unique factorizations:
\\
$2e+10=2(e+5)=2g_1, 2e+12 = (e+5)+(e+7)=g_1+g_2, 2e+14=2(e+7)=2g_2$.
\item $\Ap(\G_{e,e+5})=\{0,e+6,2e+8,\underline{2e+12},3e+12,\underline{3e+14},\underline{3e+18},3e+19,\ldots,4e+15,\underline{4e+16}\}$.
\\
Unique factorizations:
\\
 $2e+12=2(e+6)=2g_1$,
 $3e+14=(e+6)+(2e+8)=g_1+g_2$,
 $3e+18=3(e+6)=3g_1$,
 $4e+16=2(2e+8)=2g_2$.
\item $\Ap(\G_{e,e+6})=\{0,e+7,e+9,e+10,e+15,\ldots,2e+11,\underline{2e+14},\underline{2e+17},\underline{2e+18},\underline{2e+19},\underline{2e+20}\}.$
\\
Unique factorizations:
\\
$2e+14=2g_1=2(e+7), 2e+17= (e+7)+(e+10), 2e+18=2g_2=2(e+9)=g_1+g_3 $,\\
$2e+19=(e+9)+(e+10)=g_2+g_3,
2e+20=2(e+10)=2g_3$.
\end{enumerate}

We now apply Proposition  \ref{RelationsRosales} to compute $\rho(\G_{e,e+\delta})$, and show that it equals $\scC(e,e+\delta)$ in each case, concluding the proof.
Since $3 \leq e < m $, we have 
$$
{e +1 \choose 1} = e+1 \leq m = e + \delta \leq e+6 < 2(e+2) \leq {e +2 \choose 2},
$$
thus, in \eqref{EqDefCem}, we have $r = 2$, $s = m-(e+1) = \delta -1$,
and $\scC(e, e+\delta) = {e+1 \choose 2} + (\delta-1)^{\langle 2 \rangle} - (\delta-1)$.
	
\begin{enumerate}
\item 
$\Ap(\G_{e,e+1})=\{0,g_1,\ldots,g_e\}$, 
thus
$\mathrm{Min}\big(\mathcal{A}(\G_{e,e+2})\big) =
\big\{\bb_i+\bb_j\, \mid \, 1 \le i \le j \le e \big\}
$ and  $\rho(\G_{e,e+3})=\binom{e+1}{2}.$
We have $\delta-1 = 0  $, 
so 
 $(\delta-1)^{\langle 2 \rangle} = 0$ and 
$
\scC(e,e+2) =  {e+1\choose 2}.
$		

\item $\Ap(\G_{e,e+2})=\{0,g_1,\ldots,g_e,2g_1\}$, thus
\begin{align*}
\mathrm{Min}\big(\mathcal{A}(\G_{e,e+2})\big) =&
\big\{\bb_i+\bb_j\, \mid \, 1 \le i \le j \le e, (i,j) \neq (1,1) \big\} \cup \big\{
3\bb_1\big\}
\end{align*}
and we obtain $\rho(\G_{e,e+3})=\binom{e+1}{2}-1+1 = {e+1\choose 2}.$
We have $\delta-1 = 1  = {2 \choose 2 } $, 
so 
 $(\delta-1)^{\langle 2 \rangle} = {3 \choose 3 } =3$ and 
$
\scC(e,e+2) =  {e+1\choose 2}+1 -1 = {e+1\choose 2}.
$

\item 
$\Ap(\G_{e,e+3})=\{0,g_1,\ldots,g_e,2g_1,g_1+g_3\}$, thus
\begin{align*}
\mathrm{Min}\big(\mathcal{A}(\G_{e,e+3})\big) =&
\big\{\bb_i+\bb_j\, \mid \, 1 \le i \le j \le e, (i,j) \neq (1,1),(1,3) \big\} \cup \big\{
3\bb_1,2\bb_1+\bb_3\big\}
\end{align*}
and we obtain $\rho(\G_{e,e+3})=\binom{e+1}{2}-2+2 = {e+1\choose 2}.$
We have $\delta-1 = 2  = {2 \choose 2 } + {1 \choose 1}$, 
so 
 $(\delta-1)^{\langle 2 \rangle} = {3 \choose 3 }+ {2\choose 2} =2$ and 
$
\scC(e,e+3) =  {e+1\choose 2}+2 -2 = {e+1\choose 2}.
$

\item $\Ap(\G_{e,e+4})=\{0,g_1,\ldots,g_e,2g_1,g_1+g_2,2g_2\}$, thus
\begin{align*}
\mathrm{Min}\big(\mathcal{A}(\G_{e,e+4})\big) =&
\big\{\bb_i+\bb_j\, \mid \, 1 \le i \le j \le e, (i,j) \neq (1,1),(1,2),(2,2) \big\} \\
&\cup \big\{
3\bb_1,3\bb_2,2\bb_1+\bb_2,\bb_1+2\bb_2\big\}
\end{align*}
and we obtain $\rho(\G_{e,e+4})=\binom{e+1}{2}-3+4 = {e+1\choose 2}+1.$
We have $\delta-1 = 3  = {3 \choose 2 } $, 
so \\
 $(\delta-1)^{\langle 2 \rangle} = {4 \choose 3 }=4$ and 
$
\scC(e,e+4) =  {e+1\choose 2}+4 -3 = {e+1\choose 2}+1.
$

\item $\Ap(\G_{e,e+5})=\{0,g_1,\ldots,g_e,2g_1,g_1+g_2,3g_1,2g_2\}$, thus
\begin{align*}
\mathrm{Min}\big(\mathcal{A}(\G_{e,e+5})\big) =&
\big\{\bb_i+\bb_j\, \mid \, 1 \le i \le j \le e, (i,j) \neq (1,1),(1,2),(2,2) \big\} \\
&\cup \big\{
4\bb_1,3\bb_2,2\bb_1+\bb_2,\bb_1+2\bb_2\big\}
\end{align*}
and we obtain $\rho(\G_{e,e+5})=\binom{e+1}{2}-3+4 = {e+1\choose 2}+1.$
We have $\delta-1 = 4  = {3 \choose 2 } + {1\choose 1}$, 
so \\
 $(\delta-1)^{\langle 2 \rangle} = {4 \choose 3 } + {2\choose 2} = 5$ and 
$
\scC(e,e+5) =  {e+1\choose 2}+5 -4 = {e+1\choose 2}+1.
$

\item 
$\Ap(\G_{e,e+6})=\{0,g_1,\ldots,g_e,2g_1,g_3+g_1,2g_2,g_2+g_3,2g_3\}$, 
thus 
\begin{align*}
\mathrm{Min}\big(\mathcal{A}(\G_{e,e+6})\big) =&
\big\{\bb_i+\bb_j\, \mid \, 1 \le i \le j \le e, (i,j) \neq (1,1),(2,2),(1,3),(2,3),(3,3)\big\} \\
&\cup \big\{3\bb_1,3\bb_2,3\bb_3,2\bb_1+\bb_3,\bb_1+2\bb_3,2\bb_2+\bb_3,\bb_2+2\bb_3\big\}
\end{align*}
 and we obtain $\rho(\G_{e,e+6})=\binom{e+1}{2}-5+7 = \binom{e+1}{2}+2$.
We have $\delta-1 = 5  = {3 \choose 2 } + {2\choose 1}$, 
so 
\\
$(\delta-1)^{\langle 2 \rangle} = {4 \choose 3 } + {3\choose 2} = 7$ and 
$
\scC(e,e+6) =  {e+1\choose 2}+7 -5 = {e+1\choose 2}+2.
$
\end{enumerate}
\end{proof}

\begin{remark}
	Theorem \ref{TheoremSharpEM6} is sharp, in the sense that its statement is false for $m-e=7$,
	see Example \ref{Ex310} and Proposition \ref{Prop411and512}.
\end{remark}

\section{Regions where the bound is not sharp}\label{SectionRelationNonSharp}

This section is devoted to proving the second part of Theorem \ref{MainTheoremNumberRelations}. 
More generally, the goal is to identify regions of $\NN^2$ where the upper bound $\scR(e,m)\leq \scC(e,m)$ is never attained.
The methods are based on ideas from additive combinatorics and on the techniques of Section \ref{SectionToolbox}. 

Let $G$ be an abelian group. 
We use additive notation throughout. 
Let $A,B \subseteq G$ be subsets, and $k \in \NN$.
The sumset of $A$ and $B$ is the subset $A+B = \{a+b \,\mid \, a \in A, b \in B\}\subseteq G$, and  the $k$-th multiple of $A$ is the subset 
$$
kA = \underbrace{A+A+\cdots+A}_{k \text{ times}} = \left\{ \sum_{i} \lambda_i a_i\, \mid a_i \in A, \lambda_i \in \NN, \sum_i \lambda_i = k\right\}\subseteq G.$$
If $|A|=e+1$, then we have $|kA| \leq {e+k \choose k}$, and the equality $|kA| = {e+k \choose k}$ holds if and only if every element $a \in kA$ has a unique representation as a sum of $k$ elements of $A$.

We prove a combinatorial lemma, which gives a sufficient condition for a multiple of a subset to have sub-maximal size.

\begin{lemma}\label{LemmaSumsetSizekA}
Let $G$ be an abelian group with $|G|=m$, and let
 $A \subseteq G$ be a subset with $|A|= e+1$. 
	If  $m < \binom{e+k-1}{k-1}+ e\binom{e+k-2}{k-1}$ for some $k > 0$, then $|kA| < \binom{e+k}{k}$.
\end{lemma}
\begin{proof}
Consider the subset $B=(k-1)A\subseteq G$, 
then we have $|B| \leq {e+k-1 \choose k-1}$.
If $|B| < \binom{e+k-1}{k-1}$, then clearly we have $|A+B|=|kA| < \binom{e+k}{k}$. 
Now, assume  that $|B|=\binom{e+k-1}{k-1}$.
Then, every element of $B$ has a unique representation as a sum of $k-1$ elements of $A$.
This also implies that $|(k-2)A|=\binom{e+k-2}{k-2}$.
 Consider the sets 
\begin{align*}
(A\times B)^* &= \big\{(a,b) \,\mid \, a \in A, b \in B, b-a \not \in (k-2)A\big\}\subseteq G \times G,\\
(B-A)^* &= \big\{ b-a \, \mid \, (a,b) \in (A\times B)^*\big\}\subseteq G \setminus (k-2)A.
\end{align*}
Label the elements of $A = \{a_0, a_1, \ldots, a_e\}$. 
For each $a_h \in A$ and  $b = \sum_{i=0}^e \lambda_i a_i \in B$,
by the uniqueness of the representation of $b$,
we have $b-a_h\in (k-2)A$ if and only if $\lambda_h>0$.
That is, $(a_h,b) \in (A\times B )^*$ if and only if $\lambda_h = 0$,
so $b \in (k-1)A$ is a sum of the remaining $e$ elements of $A$.
Therefore, for every $a \in A$, there are exactly ${e+(k-1)-1\choose k-1}$ elements of $B$ such that $(a,b) \in  (A \times B)^*$, yielding 
\begin{align*}
|(A \times B)^*| &=(e+1)\binom{e+k-2}{k-1} =e\binom{e+k-2}{k-1}+\binom{e+k-2}{k-1}
\\
&= e\binom{e+k-2}{k-1}+\binom{e+k-1}{k-1} -\binom{e+k-2}{k-2}
\\
&= e\binom{e+k-2}{k-1}+\binom{e+k-1}{k-1} -|(k-2)A|
\\
&> m - |(k-2)A| \geq |(B-A)^*|.
\end{align*}
Since $|(A \times B)^*| > |(B-A)^*|$,
there must be two distinct pairs  $(a_h,b'), (a_{j},b) \in (A\times B)^*$ such that  $b'-a_h = b-a_j$,  that is, $a_h+b=a_j+b'$. 
Since the pairs are distinct, we have $a_h \ne a_j$ and $b\ne b'$.
Writing $b = \sum_{i=0}^e \lambda_i a_i$ and $b' = \sum_{i=0}^e \lambda'_i a_i$,
since $b-a_j \notin (k-2)A$, we must have $\lambda_j = 0$.
It follows that  $a_h +\sum_{i=0}^e \lambda_i a_i =a_j +\sum_{i=0}^e \lambda'_i a_i$ are distinct representations  of the same element of $kA$,
hence, $|kA| < \binom{e+k}{k}$ as desired.
\end{proof}

Exploiting Lemma \ref{LemmaSumsetSizekA}, we show that existence of regions in $\NN^2$ where the bound $\scR(e,m) \leq \scC(e,m)$ is not attained.

\begin{prop}\label{PropNotSharpRange}
If $\binom{e+k}{k} \le m \le \binom{e+k-1}{k-1}+ e\binom{e+k-2}{k-1}-1$
for some $k \geq 2$,
 then $\scR(e,m) < \scC(e,m)$.
\end{prop}

\begin{proof}
Let $G=\ZZ/m\ZZ$ be the additive group of  integers modulo $m$, 
let  $S = \kk[x_1, \ldots, x_e]$ and $m' =\binom{e+k}{k}$.
Since $\ell(S/\mm^{k+1})= m'$, if follows that $\mC(m',S) = \mm^{k+1}$.
Since $m \ge m'$, we get  $\mC(m,S) \subseteq \mm^{k+1}$ by Proposition \ref{PropInclusionHyperplane} (1).

Let $\G=\langle g_0,\ldots,g_e \rangle$  be a numerical monoid with $\mult(\G) = g_0=m$ and $\edim(\G) = e+1$.
Consider the subset $A = \{\overline{0},\overline{g_1},\ldots,\overline{g_e}\} \subseteq G$,
where $\bar{\cdot}$ denotes the congruence class modulo $m$. 
By assumption we have  $m <  \binom{e+k-1}{k-1}+ e\binom{e+k-2}{k-1}$, hence  $|kA| < \binom{e+k}{k}$ by Lemma \ref{LemmaSumsetSizekA}.
It follows that some element of $kA$ can be represented in two different ways as a sum of $k$ elements of $A$.
Removing the occurrences of $\bar{0}$ in these representations and lifting the representations to $\ZZ$,
we obtain an equation
\begin{equation}\label{EqDoubleRepresentation}
g_{i,1} + \cdots + g_{i,{k_1}} = g_{j,1} + \cdots + g_{j,k_2} +  a m
\end{equation}
where $g_{i,h},g_{j,h}\in \{g_1, \ldots, g_e\}$ for all $h$, $a \in \NN$, and $k_1, k_2 \leq k$.
Let $I = I_\G^{\ *}$ be as in Subsection \ref{SubsectionMonoidAlgebras}.
Comparing with \eqref{EqIdealIstar}, we see that  \eqref{EqDoubleRepresentation} gives rise  to either a monomial or a binomial in $I$, 
of degree equal to $\min(k_1,k_2+a) \leq k_1 \leq k$.
We deduce that $[I]_k \ne 0$ and,
applying Proposition \ref{PropLowerDegreeStrictInequality} with $r = k+1$, 
we conclude that $\beta_0^S(I) < \beta_0^S(\mC(m,S)) = \scC(e,m)$,
concluding the proof.
\end{proof}

\begin{example}\label{Ex310}
Applying Proposition \ref{PropNotSharpRange} with $e=3, m = 10, k =2$, we find that 
$\scR(3,10) < \scC(3,10)$, thus, Theorem \ref{TheoremSharpEM6} cannot be further extended.
\end{example}

The following theorem is the main result of this section.

\begin{thm}\label{TheoremNotSharpAsymptotic}
For any fixed value of $e \ge 3$, we have $\scR(e,m) < \scC(e,m)$ for $m \gg 0$.
\end{thm}

\begin{proof}
Let $t$ be the integer $t =\lceil 3+3\sqrt{e}\rceil$, and let $m \ge \binom{e+t}{t}$. 
Let $k$ be the unique integer such that $\binom{e+k}{k} \le m < \binom{e+k+1}{k+1}$.
In particular, we have $k \geq t \geq 9$.

The inequality
\begin{equation}\label{Ineq1}
(e-2)k^2-(2e-1)k-(e^2-1) > 0
\end{equation}
holds if $k > \frac{2e-1 + \sqrt{4e^3-4e^2-8e+9}}{2(e-2)}$.
Then, \eqref{Ineq1}  is satisfied in our assumptions, since
$$
k \geq t \geq 3+3\sqrt{e} = \frac{2e+2\sqrt{e^3}}{\frac{2e}{3}} > \frac{2e+\sqrt{4e^3}}{2(e-2)} > \frac{2e-1+\sqrt{4e^3-4e^2-8e+9}}{2(e-2)},
$$
where we used that $\frac{2e}{3} \leq 2(e-2)$ and $-8e+9<0$ since $e \geq 3$.

The inequality  \eqref{Ineq1} is equivalent to
\begin{equation}\label{Ineq2}
\frac{(e+k+1)(e+k)(e+k-1)}{e}<\frac{(e^2+e+k-1)(k)(k+1)}{e},
\end{equation}
as can be verified easily by expanding all the products.
We now  manipulate \eqref{Ineq2} to obtain equivalent inequalities
$$
\frac{(e+k+1)(e+k)(e+k-1)}{k(k+1)e}<\frac{e+k-1}{e}+e,
$$
$$
\frac{(e+k+1)(e+k)(e+k-1)}{k(k+1)e}{e+k-2\choose k-1}<\left(\frac{e+k-1}{e}+e\right){e+k-2\choose k-1},
$$
$$
{e+k+1\choose k+1} < {e+k-1\choose k-1} + e {e+k-2\choose k-1}.
$$
Combining the last inequality with $\binom{e+k}{k} \le m < \binom{e+k+1}{k+1}$,
we have found, for each $m \geq \binom{e+\lceil 3+3\sqrt{e}\rceil}{\lceil 3+3\sqrt{e}\rceil}$,
an integer $k \geq 2$ such that 
$\binom{e+k}{k} \le m <{e+k-1\choose k-1} + e {e+k-2\choose k-1}$.
Applying Proposition \ref{PropNotSharpRange}, we obtain the statement of the theorem.	
\end{proof}

Roughly speaking, Theorem \ref{TheoremNotSharpAsymptotic} says that in many cases the bound $\scR(e,m) \leq \scC(e,m)$ fails to be an equality, because the additive combinatorics of the numerical monoid forces a polynomial relation of low degree.
In particular, the ideal $I_\G^{\ *}$  cannot attain the same Hilbert function as the extremal ideal $\mC(e,m)$.
We show next that, in general, the situation is even more complicated:
it is possible that $I_\G^{\ *}$  has the same Hilbert function as $\mC(e,m)$,
but, nevertheless, the bound $\scR(e,m) \leq \scC(e,m)$ is not achieved.
This phenomenon is related to  a subtler piece of data, namely, 
the degrees of a regular sequence contained in $I_\G^{\ *}$, and,
thus, to the EGH Conjecture.

First, we need another  lemma in additive combinatorics.

\begin{lemma}\label{LemmaCplusC}
Let $\G=\langle g_0,\ldots,g_e\rangle$ be a numerical monoid. 
Let $C \subseteq \{g_1,\ldots,g_e\}$ be a subset and let $t = |C|$.
If $\Ap(\G) = \{0,g_1,\ldots,g_e\} \cup (C+C)$ and $g_0 = e+1+\binom{t+1}{2}$,
then $g_0 > t(t+1)$.
\end{lemma}

\begin{proof}
Let $m= g_0$ and 
assume by contradiction that $m \leq t(t+1)$, 
that is, $m < t^2+t+1 = {t+1 \choose 1} + t{t \choose 1}$.
Let $A\subseteq G = \ZZ/m\ZZ$ be the subset consisting of $0$ and the residues of the elements of $C$,
so that $|A|=t+1$.
Observe that $ \{0,g_1,\ldots,g_e\} \cap (C+C) = \emptyset$, since $g_1,\ldots,g_e$ are  minimal generators.
We deduce that  $|C+C|=|\Ap(\G)| - (e+1)=\binom{t+1}{2}$, and,
since the union $A+A = \{0\} \cup C \cup (C+C)$ is disjoint, that $|A+A|=\binom{t+2}{2}$. 
This contradicts Lemma \ref{LemmaSumsetSizekA} (taking $k=2$ and substituting $t$ with $e$). 
\end{proof}

The following result covers the first two  cases that are not covered by our previous results.
The intricate argument gives an indication of how complicated Problem \ref{ProblemNumberRelations} becomes in general.

\begin{prop}\label{Prop411and512}
We have
	$\scR(4,11) < \scC(4,11)$ and $\scR(5,12) < \scC(5,12)$.
\end{prop}

\begin{proof}
Let $e \in \{4,5\}$ and $m = e+7$.
Let $I = I_\G^{\ *}\subseteq S$ be as in Subsection \ref{SubsectionMonoidAlgebras}, and let $J$ be a  monomial initial ideal  of $I$, with respect to some term order.
Then, $\HF(I) = \HF(J)$,
 $\ell(S/I) = \ell(S/J) = m$,
 and $\beta_{i,j}^S(I) \leq \beta_{i,j}^S(J) $ for all $i,j$,
cf. \cite[Corollaries 3.3.3 and 6.1.5]{HerzogHibi}.

Let $C = \mC(m,S)$.
By Lemma \ref{Lemma1e6}, 
if $\HF(I) \ne \HF(C)$ then $\beta_0^S(I) < \beta_0^S(C) = \scC(e,m)$.
Thus, from now on, we assume that  $\HF(I) = \HF(C)$.

Since $m = e+7$,
it is easy to check that
$C = (x_1, \ldots, x_{e-3})(x_1, \ldots, x_e) + \mm^3$.
In other words, $C$ contains all the cubics, 
and all the quadratic monomials except  the six  involving only  $x_{e-2},x_{e-1},x_e$.
It follows that  
the graded component  $[C]_2$ generates an ideal  $([C]_2)\subseteq S$ of codimension $e-3$.

Let $C' \subseteq S$ be another monomial ideal such that $\HF(C') = \HF(C)$ and such that 
the graded component $[C']_2$ generates an ideal $C'' = ([C']_2)\subseteq S$ of codimension $\mathrm{codim}(C'')=s\leq e-3$.
We claim that there exists a permutation of the variables that transforms $C'$ into  $C$.
Let $P\supseteq C''$ be a minimal prime of $C''$ with $\mathrm{codim}(P)=s$.
 Since $C''$ is a monomial ideal, $P$ is a monomial ideal too, 
 and, up to permutation of the variables, we have $P = (x_1, \ldots, x_s)$.
 It follows that the six quadrics in  $x_{e-2},x_{e-1},x_e$ do not belong to $P$, and hence to $C''$.
We conclude that  $[C']_2=[C'']_2 \subseteq [C]_2$.
On the other hand, since $\HF(C)=\HF(C')$ and $\mm^3 \subseteq C \subseteq \mm^2$,
we deduce that $[C']_d = [C]_d$ for all $d$,
so $C'=C$ as desired.

Now, we show that there exists no permutation of the variables that transforms $J$ into $C$.
Assume the contrary.
Then, there exists a subset $\mathcal{I}\subseteq \{1, \ldots, e\}$ with $|\mathcal{I}|=3$ such that  $x_{i}x_j \notin J$ for all $i,j \in \mathcal{I}$.
We claim that  the corresponding six elements $g_i+g_j$, where $i,j \in \mathcal{I}$,
belong to 
$\Ap(\G)$ and are all distinct.
If some $g_i+g_j \notin \Ap(\G)$, then
it follows from \eqref{EqIdealIstar} that $x_ix_j \in I$,
thus $x_ix_j \in J$, contradiction.
If $g_i + g_j = g_{i'}+g_{j'}$ for some $i,i',j,j'\in \mathcal{I}$ with $(i,j) \ne (i',j')$,
then it follows from \eqref{EqIdealIstar} that $x_ix_j-x_{i'}x_{j'} \in I$,
thus either $x_ix_j\in J$ or $x_{i'}x_{j'} \in J$,
contradiction.
Since $\Ap(\Gamma)$ contains exactly $m-e-1 = 6$ elements that are positive and not minimal generators, 
it follows that  $\Ap(\G) = \{0,g_1,\ldots,g_e\} \cup (C+C)$ with $C =\{ g_i \, : \, i \in \mathcal{I}\}$.
Applying Lemma \ref{LemmaCplusC} with $t=3$, we find that $m = g_0 > 12$, contradiction.
This concludes the proof of the claim that no permutation transforms $J$ into $C$.

By the previous two paragraphs, 
we have proved that $J$ contains a regular sequence of quadrics of length at least $e-2$. 
Recall that $e \leq 5$.
The EGH Conjecture is known to hold when the regular sequence consists of quadrics and the polynomial ring has at most 5 variables \cite[Theorem 4.1]{GH}.
We conclude that the total Betti numbers of $J$, and hence those of $I$, 
are bounded above by those of the unique ideal $H = (x_1^2, x_2^2, \ldots, x_{e-2}^2) +L$ such that $L\subseteq S$ is a lexsegment ideal and $\HF(H) = \HF(I)$.
In both cases $e=4,5$ we verify by direct calculation on \cite{M2} that 
$\beta_i^S(H) <\beta_i^S(C)$ for all $i = 0,\ldots, e-1$.
Using \eqref{InequalityBettis} and Proposition \ref{PropFirstLastBettiValla}, we conclude that $\rho(\G) \leq \beta_0^S(I)\leq \beta_0^S(H)  <\beta_i^S(C) = \mC(e,m)$.
\end{proof}

\section{Results for the type}\label{SectionType}

We conclude the paper with direct applications of the results of Sections \ref{SectionToolbox}, \ref{SectionRelationsSharp}, and \ref{SectionRelationNonSharp},  to the type of numerical monoids.

First, we discuss the cases where the bound $\scT(e,m) \leq \scD(e,m)$ is attained.
Applying 
Propositions \ref{PropFirstLastBettiValla} and \ref{PropRigidity}, we deduce the following two results from  Theorems \ref{TheoremSharpLargeM} and \ref{TheoremSharpEM6}, 
respectively.

\begin{cor}
Let $e,m \in \NN$ with $3 \leq e < m$.
For each fixed value of $m-e$, we have 
  $\scT(e,m) = \scD(e,m)$ for all $m \gg 0$.
\end{cor}

\begin{cor}
Let $e, m \in \NN $ with $3 \leq e < m$.
If $m - e \leq 6$,
then  $\scT(e,m) = \scD(e,m)$.
\end{cor}

Now, we switch to the cases where the bound is not attained.
Unfortunately, we are unable to obtain almost any result in this  regard.
The main complication, as observed in Section \ref{SectionToolbox},
is the fact that  finding sufficient conditions for having strict inequalities  in Theorem \ref{TheoremValla} 
is harder if $i\geq 2$,
see Proposition \ref{PropLowerDegreeStrictInequality} and Remark \ref{RemarkNonStrictHigherBetti}.
Nevertheless, we believe that the function $\scT(e,m)$ does behave like $\scR(e,m)$, in particular, we conjecture that the analogue of Theorem \ref{TheoremNotSharpAsymptotic} holds for the type.

\begin{conjecture}
For any fixed value of $e \ge 3$, we have $\scT(e,m) < \scD(e,m)$ for $m \gg 0$.
\end{conjecture}

A second complication is the fact that the LPP Conjecture is known to hold in far fewer cases than the EGH Conjecture.
For example, we can only extend the first of the two results of Proposition \ref{Prop411and512} to the type.
While this particular fact is certainly not very significant by itself,
it does suggest that Problem \ref{ProblemType} may be even harder than Problem \ref{ProblemNumberRelations} in general.

\begin{prop}\label{PropType411}
We have
	$\scT(4,11) = \scD(4,11)-1$.
\end{prop}
\begin{proof}
Since the type of a numerical monoid clearly does not depend on the choice of the field $\kk$, we may assume that the characteristic of $\kk$ is 0.
The proof of the inequality $\scT(4,11) < \scD(4,11)$ proceeds exactly as in Proposition \ref{Prop411and512}.
The only difference is that, instead of the EGH Conjecture,
we use the LPP Conjecture, which is known to hold for regular sequences of 2 quadrics in characteristic 0 by \cite[Main Theorem]{CS18},
to obtain 
$\type(\G) \leq \beta_{e-1}^S(I)\leq \beta_{e-1}^S(H)  <\beta_{e-1}^S(C) = \scD(e,m)$.
In order to conclude that 	$\scT(4,11) = \scD(4,11)-1$,
it suffices to show an example of a numerical monoid $\G$ with $\edim(\G) = 5, \mult(\G) =11$, and 
$\type(\G) =\scD(4,11)-1 = 6$.
The monoid $\G = \langle 11,12,14,15,20 \rangle$ is such an example.
\end{proof}

\section*{Acknowledgments}
We are grateful to Giulio Caviglia and Alessio D'Alì for  helpful comments.
Some computations were performed with the softwares GAP \cite{GAP} and
 Macaulay2 \cite{M2}.

\section*{Funding}
Moscariello was supported by 
the grant “Proprietà locali e globali di
anelli e di varietà algebriche” PIACERI 2020–22, Università degli Studi di Catania.
Sammartano was  supported by the grant PRIN 2020355B8Y
{\em Square-free Gr\"obner degenerations, special varieties and related topics}
and by the INdAM – GNSAGA Project CUP E55F22000270001.

\end{document}